\documentclass[final,onefignum,onetabnum]{siamart190516}

\usepackage{braket,amsfonts}

\usepackage{array}

\usepackage[caption=false]{subfig}
\usepackage{pgfplots}

\newsiamthm{claim}{Claim}
\newsiamremark{remark}{Remark}
\newsiamremark{hypothesis}{Hypothesis}
\crefname{hypothesis}{Hypothesis}{Hypotheses}

\usepackage{algorithmic}

\usepackage{graphicx,epstopdf}

\Crefname{ALC@unique}{Line}{Lines}

\usepackage{amsopn}

\usepackage{xspace}
\usepackage{amsmath}
\usepackage{amssymb}
\usepackage{bold-extra}
\usepackage[most]{tcolorbox}

\colorlet{texcscolor}{blue!50!black}
\colorlet{texemcolor}{red!70!black}
\colorlet{texpreamble}{red!70!black}
\colorlet{codebackground}{black!25!white!25}

\lstdefinestyle{siamlatex}{%
  style=tcblatex,
  texcsstyle=*\color{texcscolor},
  texcsstyle=[2]\color{texemcolor},
  keywordstyle=[2]\color{texemcolor},
  moretexcs={cref,Cref,maketitle,mathcal,text,headers,email,url},
}

\tcbset{%
  colframe=black!75!white!75,
  coltitle=white,
  colback=codebackground, 
  colbacklower=white, 
  fonttitle=\bfseries,
  arc=0pt,outer arc=0pt,
  top=1pt,bottom=1pt,left=1mm,right=1mm,middle=1mm,boxsep=1mm,
  leftrule=0.3mm,rightrule=0.3mm,toprule=0.3mm,bottomrule=0.3mm,
  listing options={style=siamlatex}
}

\newtcblisting[use counter=example]{example}[2][]{%
  title={Example~\thetcbcounter: #2},#1}

\newtcbinputlisting[use counter=example]{\examplefile}[3][]{%
  title={Example~\thetcbcounter: #2},listing file={#3},#1}

\DeclareTotalTCBox{\code}{ v O{} }
{ 
  fontupper=\ttfamily\color{black},
  nobeforeafter,
  tcbox raise base,
  colback=codebackground,colframe=white,
  top=0pt,bottom=0pt,left=0mm,right=0mm,
  leftrule=0pt,rightrule=0pt,toprule=0mm,bottomrule=0mm,
  boxsep=0.5mm,
  #2}{#1}

\patchcmd\newpage{\vfil}{}{}{}
\begin{tcbverbatimwrite}{tmp_\jobname_header.tex}
\title{A Chebyshev multidomain adaptive mesh method for reaction-diffusion equations}

\author{Jae-Hun Jung \thanks{Department of Mathematics, POSTECH Mathematical Institute for Data Science (MINDS), Pohang University of Science and Technology, Korea (\email{jung153@postech.ac.kr}).}
\and Daniel Olmos-Liceaga\thanks{Department of Mathematics, University of Sonora, M\'exico (\email{daniel.olmos@unison.mx}).}}

\end{tcbverbatimwrite}
\input{tmp_\jobname_header.tex}

\ifpdf
\hypersetup{ pdftitle={Guide to Using  SIAM'S \LaTeX\ Style} }
\fi

\begin{document}
\maketitle

\begin{tcbverbatimwrite}{tmp_\jobname_abstract.tex}
\begin{abstract}
Reaction-Diffusion equations can present solutions in the form of traveling waves. Such solutions
evolve in different spatial and temporal scales and it is desired to construct numerical
methods that can adopt a spatial refinement at locations with large gradient solutions.
In this work we develop a high order adaptive mesh method based on Chebyshev polynomials with a multidomain approach for the traveling wave solutions of
reaction-diffusion systems, where the proposed method
uses the non-conforming and non-overlapping
spectral multidomain method with the temporal adaptation of the computational mesh. Contrary to the existing multidomain spectral methods for reaction-diffusion equations, the proposed multidomain spectral method solves the given PDEs in each subdomain locally first and the boundary and interface conditions are solved in a global manner. In this way, the method can be parallelizable and is efficient for the large reaction-diffusion system. We show that the proposed method is stable and provide both the one- and two-dimensional numerical results that show the efficacy of the proposed method.
\end{abstract}


\end{tcbverbatimwrite}
\input{tmp_\jobname_abstract.tex}

\section{Introduction}
\label{sec:intro}

\noindent Reaction diffusion systems have been a very active area of research for many decades. They arise in areas such as
population dynamics and epidemiology \cite{brauer}, physiology
\cite{keener_book} and biology \cite{mu01}. Classic examples are the modelling of animal coat patterns \cite{mu01}, the
Belousov-Zhabotinsky (BZ) reaction \cite{epstein,kuramoto,mu01,ty94}, the Hodgkin-Huxley model of the
propagation of the action potential along nerve cells \cite{keener_book}, and models of disease propagation in
an ecosystem \cite{brauer}. The general form of a reaction diffusion (RD) system with space-independent diffusion is given by
\begin{equation}\label{RD_pde}
\begin{array}{lcl}
\frac{\partial \bf{V} }{\partial t}&=& \mathbf{D} \nabla^2 \bf{V} + \bf{F}(\bf{V})\\
\end{array}
\end{equation}
where $\mathbf{V}=(V_1,V_2,...,V_N)$ with $V_i=V_i(\mathbf{x},t)$, $(i=1,\cdots, N)$ can represent the concentrations of $N$ chemical species
in a reaction \cite{epstein}, or the number of susceptible, infected and recovered individuals  at position
$\mathbf{x}$ and time $t$ for a model in epidemiology \cite{brauer}. The function $\bf F$ is called the reaction term
and models the local dynamics due to the interactions among $V_i$. The spatial variation of $V_i(\mathbf{x},t)$ is modelled
with the diffusion term $\mathbf{D} \nabla^2 \bf{V}$, where the matrix $\mathbf{D}$ is the diffusion coefficient matrix.\\

\noindent The solutions of such systems may evolve in at least two time and spatial scales and their numerical
computation with uniform grids becomes expensive \cite{Cherry00}. There have been various
efforts to solve such equations accurately with as few points as possible. Many efforts for developing
such methods arise from the cardiac physiology community, where fast and reliable computations are needed to
explore phenomena such as cardiac fibrillation \cite{fe02}. In this area, there is a considerable
amount of methods that use non uniform grids in order to solve numerically equations with different spatio-temporal
solutions. For example, Cherry et. al. \cite{Cherry00} and Trangenstein et. al. \cite{tran04} use adaptive mesh
refinement methods. Krause et. al. \cite{krause2015} present an
adaptive method based on locally structured meshes. Also, finite element methods have been proposed for general
reaction diffusion equations \cite{Hu2012} and for cardiac dynamics in particular
\cite{arthurs2012,Hoermann2018,chamakuri2019}. Finally, spectral and pseudospectral methods have been also used for solving reaction-diffusion equations \cite{e83,bueno2014,jones96,bb93,zhang2020} in general,
and for cardiac dynamics in particular \cite{bueno2006,Rod2018_1,Rod2018_2}. In \cite{Rod2018_1,Rod2018_2} Chebyshev pseudospectral methods
are developed using a fixed non uniform mesh to describe the propagation of waves. \\

\noindent In the present work, we follow the idea in \cite{os09,Rod2018_1,Rod2018_2}, and propose a non-overlapping, non-conforming multidomain spectral
method for the reaction diffusion equations. One of the main differences between the method developed in this work and
\cite{os09,Rod2018_1,Rod2018_2} is that subdomains overlap in one single point for the one dimensional case while the previous works use two overlapping points. Under this approach, it becomes possible to assign different number of collocation points to each subdomain. Based on
this premise we can assign the larger number of points in regions where fast transitions occur than regions that has not rapid changes in space. This approach is not a new method in spectral community. However, to the best of authors knowledge such approach has not been applied to the traveling wave solutions of the reaction-diffusion equations. Rather there are various previous works using discontinuous Galerkin methods for the solution of the reaction-diffusion equations \cite{zhu2009,Liu2009,zhang2014}. Further, as the subdomains overlap at only one point, the given PDEs can be solved separately in each subdomain, which makes the method parallelizable. Then the interface and boundary conditions are applied in a global manner to patch solutions across subdomains. The implementation of the interface and boundary conditions results in a linear system, which can also be solved in an efficient way. In this paper, we show that the proposed approach is stable. As for the grid adaptivity, we use the uniform intervals of subdomains and change the number of points inside each interval, with time, depending on the smoothness of the solution within the subdomain. For the temporal variation of the number of grid points, we use the simple switching algorithm. That is, the subdomain increases the total number of grid points when the solution becomes stiff to a certain degree with its few adjacent subdomains adopting the same number of grid points as a buffer zone to guarantee the consistency of the solution for stability. We first construct the proposed method for the one-dimensional problem. The extension to the two-dimensional problem is straightforward, dimension-by-dimension. We present numerical solutions for both the one- and two-dimensional problems. The numerical solutions presented in this paper shows the efficacy of the proposed method.
This paper is organized as follows. In Section \ref{Section_mathmodels} we introduce the mathematical models of reaction-diffusion equations. In Section \ref{numericalmethods} we present the proposed adaptive multidomain Chebyshev spectral method to properly capture the stiff traveling wave solutions. In the same section, we prove the stability of the proposed method. In Section \ref{Sec:switch} the switching algorithm is explained. In Section \ref{Sec_numres}, we show numerical results including convergence and grid adaptivity for both one- and two-dimensional problems. Finally, we close the paper with a Section of Discussions and Conclusions.

\section{Mathematical models}\label{Section_mathmodels}

\noindent In order to test our proposed method we focus on equations that accept traveling waves. The first model belongs to the family of the Fitzhugh-Nagumo models. These types of equations are a generic representation of excitable media and are used to study propagation of action potentials in cardiac cells. The model under study was proposed by Barkley \cite{db91} and is given by the following set of equations,

\begin{equation}\label{FHNeq}
\begin{array}{lcl}
u_t & = & \delta \nabla^{2}u+\frac{1}{\epsilon}u(1-u)\left(u-\frac{v+b}{a}\right)\\
v_t & = & u-bv,
\end{array}
\end{equation}
where $u: \mathbb{R}\times \mathbb{R}^+\rightarrow \mathbb{R}$ represents the voltage across the cell membrane and $v: \mathbb{R}\times \mathbb{R}^+\rightarrow \mathbb{R}$ represents a gating variable to control the voltage \cite{keener_book}. $\epsilon \in \mathbb{R}^+$ is the time scale separation parameter between $u$ and $v$, given by the fast inward sodium current, and $a\in \mathbb{R}^+$ and $b\in \mathbb{R}^+$ are the parameters related to the action potential duration, threshold and time recovery of the cell. A detailed explanation of the dynamics for Eq. (\ref{FHNeq}) can be found in \cite{db91}.

\noindent Equation (\ref{FHNeq}) is a simplified version of more complex models in heart dynamics as those shown in \cite{fenton2008}, and preserves some important properties, such as pulse propagation and more than one scale spatiotemporal changes in the solution. In cardiac cells, the presence of pulses indicate that contracting mechanisms activate, and therefore, failure in the correct propagation of pulses are responsible for the presence of certain type of arrhythmias, particularly fibrillation \cite{fe02}. Therefore, these models are useful to
understand mechanisms of arrhythmias origin and evolution, as well as its control.\\

\noindent The second set of equations is known as the Gray-Scott model \cite{grayscott1984}, and is a theoretical model of an autocatalytic chemical reaction. The equations are given by

\begin{equation}\label{GSeq}
\begin{array}{lcl}
u_t & = & \delta \nabla^{2}u+\frac{1-u}{\tau_{r}}-uv^{2},  \\
v_t & = & \nabla^{2}v+\frac{v_{0}-v}{\tau_{r}}+uv^{2}-\kappa_{2}v,
\end{array}
\end{equation}
where $u: \mathbb{R}\times \mathbb{R}^+\rightarrow \mathbb{R}$ and $v: \mathbb{R}\times \mathbb{R}^+\rightarrow \mathbb{R}$ represent the concentration of reactant $u$ and autocatalyst $v$ at time $t$, respectively. $\delta$ is the coefficient ratio of reactants $u$ and $v$, $\tau_r$ is the ratio of the reactor volume and the total volumetric flow rate, $v_0$ is related to the inflow concentration of the autocatalyst related to the inflow concentration of the reactant, and $\kappa_{2}$ is the decaying rate of the autocatalyst. A detailed study of its local dynamics can be found in \cite{pet94}.\\
\noindent The Gray-Scott model has been of interest due mainly to (i) the understanding of the great variety of patterns that can be found when some of its parameters are changed \cite{pearson93}; (ii) the richness of its dynamical system behavior \cite{Alhumaizibook} and (iii) its help to elucidate the underlying processes in some auto-catalytic chemical reactions \cite{Lee94}.

\noindent Propagating pulses in the one dimensional case, can be observed \cite{pet94}, in which wave reflection and wave splitting phenomena are of interest. In two dimension, the presence of curvature in the propagating fronts leads to have pattern formation via pulse propagation \cite{Muratov2001}. \\

\noindent In general, these two models are excellent examples of pulse propagation phenomena. Although both equations lie on the family of excitable systems \cite{os09,pet94}, the traveling pulse solutions have different qualitative behaviour. Propagating pulses for the Fitzhugh-Nagumo equations when solved with impermeable conditions, lead to annihilation of the pulses at the boundary. Contrary to this, the Gray-Scott model pulses are reflected at the boundary. Similarly, Fitzhugh-Nagumo propagating pulses annihilate when they collide with each other, whereas for the Gray-Scott pulses reflection is again observed.

\noindent Finally, it has been shown that
Fitzhugh-Nagumo pulses are exponentially stable under small perturbations \cite{Yanagida1985}. However, in the case of the Gray-Scott equations, even it has been shown analitically and numerically
the existence of traveling solutions \cite{Manukian2015,pet94}, it is an open question the understanding of the stability of such pulses.

\section{Numerical methods}\label{numericalmethods}

\noindent Many RD equations like Fitzhugh-Nagumo (FHN) and Gray-Scott (GS) yield traveling wave solutions (fronts or pulses) that evolve in different spatial and temporal scales \cite{pet94,Tyson1988}. Such solutions have an exponential behavior in their state transitions and therefore, are differentiable. However, if we focus on a one dimensional propagation, with the introduction of a computational mesh and taking the width of the mesh larger than the width of the transition of the pulse, it will imply that the fast transitions behave as jump discontinuities. In order to recover the smooth behavior, it is mandatory to include more points in the regions with fast transitions. One way of doing this is to refine all the domain equally. However, this implies that we might be refining the mesh in locations where it is not needed, i.e. where there are no fast spatial transitions of the solution. A better strategy is to refine meshes only at places where the solutions are highly non-smooth. In this paper, we apply a similar approach, i.e. inhomogeneous multidomain spectral method for solving the reaction-diffusion equations. As we will explain in the following sections, the multidomain spectral method, particularly for hyperbolic type PDEs, has been developed thoroughly. For the multidomain spectral method, both the overlapping and non-overlapping approaches have been adopted for various problems. However,
to the authors' best knowledge, no non-overlapping multidomain spectral method has been applied to the reaction- diffusion equations while there are various research with the discontinuous
Galerkin methods. Non-overlapping multidomain spectral method has individual subdomains share only the domain interface(s) where the continuity of the solution and its derivative(s) is not required. In this work, we consider the non-overlapping multidomain spectral method with the continuity of the solution and its first derivative at the domain interfaces to guarantee the stability of the numerical solutions of the reaction-diffusion equations. Although the non-overlapping multidomain spectral method is not new, its application to the reaction-diffusion equations is new.

\noindent  Further, our scheme will use an adaptive mesh, such that the refined mesh will evolve together with the fast transitions. Once a fast transition leaves a particular location, the mesh will become coarse at this particular location.

\noindent In order to describe our numerical solution method, we first focus on solving Eq. (1) numerically in a one dimensional truncated domain given by $\Omega=[x_L,x_R] \subset \mathbb{R}$.
We consider the system of differential equations for $u$ and $v$ given by
\begin{equation}\label{equation2solve}
\begin{array}{lcl}
\frac{\partial u}{\partial t} &=&  \delta_u\nabla^2u + f(u,v) \\
\frac{\partial v}{\partial t} &=&  \delta_v\nabla^2u + g(u,v),
\end{array}
\end{equation}
where $\delta_u \in \mathbb{R}^+$  and $\delta_v \in \mathbb{R}^+$ are the diffusion coefficients for $u$ and $v$, respectively, and $f$ and $g$ are the local kinetics. The initial condition is given by
\begin{equation}\label{ic_num}
\left\{
\begin{array}{lcl}
u(x,0)&=&u_0(x)\\
v(x,0)&=&v_0(x)
\end{array}
\right. \qquad \qquad x \in \Omega
\end{equation}
and no flux boundary conditions are employed, such that
$$
\left.\frac{\partial u}{\partial x}\right|_{x \in \partial \Omega}=0
\qquad
\textrm{and}
\qquad
\left.\frac{\partial v}{\partial x}\right|_{x \in \partial \Omega}=0.
$$
\subsection{Chebyshev Pseudospectral Method}\label{Section_pseudo_multi}

\noindent In this section we briefly explain the basis of the method we will use to develop our adaptive method and more detailed description can be found in \cite{os05,os09}. The main idea is to consider the Chebyshev polynomials of degree at most $N$, orthogonal in the interval $[-1,1]$ respect to the weight function $\textrm{w}(z)=(1-z^2)^{-1/2}$, for the approximation of the unknown function. By using the Chebyshev-Gauss-Lobatto quadrature points $z_{i} = -\cos\left(\frac{\pi
i}{N}\right)$, ($i=\overline{0,N}$) and the weights $\omega_i = \frac{\pi}{N}$ for all $i$ except  $\omega_0 = \omega_N =\frac{\pi}{2N}$, we have that

\begin{equation}\label{lobato}
\int_{-1}^1 \textrm{w}(z)f(z)dz \simeq \sum_{i=0}^{N}w_{i}f(z_i)
\end{equation}
 where $N+1$ is the total number of quadrature points used.
Since any piecewise continuous function, $f \in L_{\textrm{w}}^2[-1,1]$ can
be expanded in a Chebyshev polynomial series that is convergent in
the mean of the $L_{\textrm{w}}^2$ norm, we have that for, $f(z) \in L^2_w[-1,1]$
\begin{equation}\label{expand}
f(z)\approx f_N(z)=\sum_{k=0}^Na_kT_k(z) , \qquad a_k=\frac{2}{c_k\pi}\int_{-1}^{1} \textrm{w}(z)f(z)T_{k}(z)dz
\end{equation}
where $T_k(z)$ is the Chebyshev polynomial of degree $k$, $a_k$ are the expansion coefficients and $c_k = \pi$ for $k = 0, N$ and $c_k =\pi/2$ otherwise. With Eq. (\ref{lobato}) and the expression for $a_k$ we obtain the interpolation of $f(z)$
\begin{equation}
f_N(z)  =\sum_{j=0}^{N}I_j(z)f(z_j)
\end{equation}
where the interpolating polynomials, $I_j(z)$, are given by
\begin{equation}\label{interp}
I_j(z)=\frac{2\nu_j}{N} \sum_{k=0}^{N} \nu_kT_k(z_j)T_k(z)\
\end{equation}
where $\nu_0=\nu_N=1/2$ with $\nu_k=1$ if $k \neq 0,N$ and
$I_j(z_i)=\delta_{ij}$ is the cardinal condition such that
$f_N(z_i)=f(z_i)$. The $n$th derivative of $f(z)$ at
the quadrature points are then given approximately by
\begin{equation}\label{nder}
f^{(n)}(z_k) \simeq f_N^{(n)}(z_k) =\sum_{j=0}^{N}I_{j}^{(n)}(z_k)f(z_j).
\end{equation}
If we let $\mathbf{f}= (f(z_0), f(z_1), \cdots, f(z_N))^T$,  Eq.
(\ref{nder}) can be rewritten as
\begin{equation}
\mathbf{f^{(n)}_N}=\mathbf{D^{(n)}\cdot f}
\end{equation}
where $\mathbf{f^{(n)}_N} \in \mathbb{R}^{N+1}$ is the approximation vector of the $n$th derivative of $f(x)$, and $\mathbf{D^{(n)}} \in \mathbb{R}^{(N+1)\times (N+1)}$ is the $n$th derivative operator with each $jk$th element given explicitly by
\begin{equation}
\label{deriv}
D^{(n)}_{jk} :=I_j^{(n)}(z_k) =  \left.\frac{d^{n}I_j(z)}{dz^n}\right|_{z=z_k}.
\end{equation}
Based on these definitions, Eq. (\ref{equation2solve}) can be reduced to a set of ODEs as below
\begin{equation}\label{fhn_ode_sys}
\begin{array}{lcl}
\frac{dU_{i}}{dt}& = & A \delta_u \sum_{k=0}^{N}D_{ik}^{(2)}U_{k} + f(U_{i},V_{i})\\
\frac{dV_{i}}{dt}& = & A \delta_v \sum_{k=0}^{N}D_{ik}^{(2)}V_{k} + g(U_{i},V_{i})\\
\end{array}
\end{equation}
where $U_{i}$ and $V_i$ are the  approximations of $u(x_{i},t)$ and $v(x_i,t)$, respectively and $A=\frac{4}{(x_R-x_L)^2}$ appears as a consequence of the
linear transformations of the physical domain $[x_L,x_R]$ to the reference domain $[-1,1]$. No-flux boundary conditions are implemented in a similar way
\begin{equation}\label{nbcx1r}
\begin{array}{lcl}
\sum_{j=0}^{N} {D^{(1)}}_{0j}U_{j}&=&0, \quad x = x_L \\
\sum_{j=0}^{N} {D^{(1)}}_{Nj}U_{j}&=&0, \quad x = x_R
\end{array}
\end{equation}
for $u$, and similarly for $v$.\\

\subsection{Non-overlapping multidomain spectral method}

In order to apply the Chebyshev pseudospectral method above in a multidomain setting, several multidomain methods have been developed, mainly the overlapping multidomain methods for the reaction-diffusion eqeuations, including the Chebyshev multidomain (CMD) method \cite{os05}. In \cite{os05} the interval $[x_L,x_R]$ is divided into $M$ overlapping
subintervals, $I_{\mu}=[x_{0}^{\mu}, x_{N}^{\mu}]$, with  $\mu=1,\cdots, M$, and  all the subintervals have the same length and the same number of collocation points. For each
subinterval, the procedure described in Eqs. (\ref{expand}) to (\ref{nder}) is applied with the resulting system of
coupled ODEs given by Eq. (\ref{fhn_ode_sys}) with $A=\frac{4}{\left(x_{N}^{\mu}-x_{0}^{\mu}\right)^2}$. The first and second derivative matrices $\mathbf{D}^{(1)}$ and  $\mathbf{D^{(2)}}$ in Eqs.
 (\ref{fhn_ode_sys}) and (\ref{nbcx1r}), respectively  become block diagonal matrices as shown in \cite{os05}.

\noindent  One disadvantage of the CMD method \cite{os05} is that all the subdomains or subintervals need to have the same number of collocation points. It is possible to consider the overlapping multidomain method with each subdomain possibly having different number of quadrature points. In that case, extra interpolation/extrapolation procedures are required to find the solution values off the quadrature points. Although it is doable, the computational complexity becomes large, particularly when two- or three-dimensional problems are considered. Instead of using the overlapping multidomain method, we consider the non-overlapping multidomain method where any intersection of the adjacent subdomains is the set of one point at the domain interface for one-dimensional problem, of edges for two-dimensional problem and so forth.
Figure \ref{md2}, shows the schematic illustration of the non-overlapping multidomain for one-dimension where we have $M$ subdomains, $I_i, i = 1, \cdots, M$ with each subdomain $I_k$ of $N_k + 1$ quadrature points. Therefore, it follows that we can have subdomains with more points than other subdomains in places where the fast transitions in solution occur.


\begin{figure}[h]
\begin{center}
\begin{tikzpicture}
\draw [color = black!99] (1,0) -- (4,0);
\draw [color = black!99] (4,0) -- (7,0);
\draw [densely dashed, color = black!99] (7,0) -- (10,0);
\draw [color = black!99] (10,0) -- (13,0);

\draw [color = black!99] (1,-0.5) -- (1,0.5);
\draw [color = black!99] (4,-0.5) -- (4,0.5);
\draw [color = black!99] (7,-0.5) -- (7,0.5);
\draw [color = black!99] (10,-0.5) -- (10,0.5);
\draw [color = black!99] (13,-0.5) -- (13,0.5);

\node at (1,-1) (nodeXi) {$x^1_0$};
\node at (3.9,-1) (nodeXi) {$x^1_{N_1} = x^2_0$};
\node at (6.9,-1) (nodeXi) {$x^2_{N_2} = x^3_0$};
\node at (9.9,-1) (nodeXi) {$x^{M-1}_{N_{M-1}} = x^M_0$};
\node at (12.9,-1) (nodeXi) {$x^M_{N_M}$};

\node at (2.5,0.4) (nodeXi) {$\mbox{I}_1$};
\node at (5.5,0.4) (nodeXi) {$\mbox{I}_2$};
\node at (8.5,0.4) (nodeXi) {$\cdots\cdots$};
\node at (11.5,0.4) (nodeXi) {$\mbox{I}_M$};

\node at (7,2) (nodeXi) {$\mbox{Subdomains:}\quad \mbox{I}_{i}, i = 1, \cdots, M$};

\draw[->] (5,1.7) -- (2.7,0.8);
\draw[->] (7,1.7) -- (11,0.7);
\draw[->] (6,1.7) -- (5.5,0.7);
\end{tikzpicture}
\end{center}
	\caption{Non-overlapping multidomain in one dimension. Each subdomain intersects with its adjacent subdomain only at the domain interface.}
	\label{md2}
\end{figure}
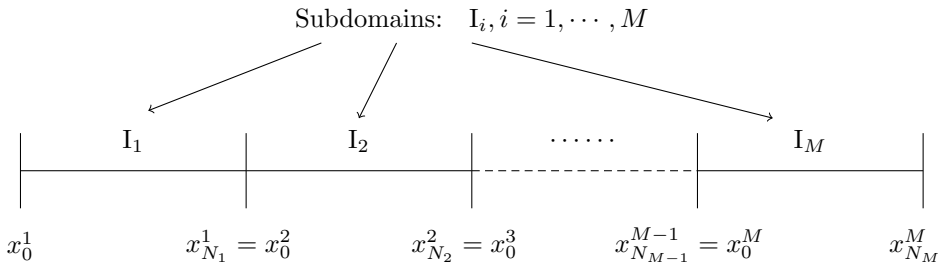

\subsubsection{Calculation of the diffusion term}

Since each subdomain has different number of quadrature points, we need to define the second derivative operator for each subdomain. As each subdomain meets its adjacent subdomain at the interface it follows that the second derivative at the interface is calculated twice. At the end of the calculation, when we implement the boundary and interface conditions, the values of the approximated solutions at the domain  interface become equal. The use of this approach leads to a parallelizable algorithm which makes its computation more efficient. Let $\mathbf{D}^{1,k}$ and $\mathbf{D}^{2,k}$ be the first and second order derivatives defined in $k$th subdomain with the total $N_k + 1$ quadrature points, $x^k_0, \cdots,
x^k_{N^k}$ in it, i.e.
$$
\mathbf{D}^{1,k} = \frac{2}{{x^k_{N_k}}-x^k_0} {\bf D}^{(1)} \qquad \textrm{and}\qquad
\mathbf{D}^{2,k} = \frac{4}{(x^k_{N_k}-x^k_0)^2} {\bf D}^{(2)}
$$
where $\mathbf{D}^{(1)}$ and $\mathbf{D}^{(2)}$ are defined in Eq. \ref{deriv}. Here note that $\mathbf{D}^{(1)}$ and $\mathbf{D}^{(2)}$ are defined with $N_k$ corresponding to $I_k$ subdomain whose interval is $[x_0^k, x^k_{N_k}]$.
In this case, the global second order derivative matrix $\mathbf{D}^{2}$ becomes
\begin{equation}\label{matrixD2}
\mathbf{D}^{2}=\left(
\begin{array}{llllll}
\mathbf{D}^{2,1}&0&0&\dots&0\\
0&\mathbf{D}^{2,2}&0&&0\\
\vdots  & &\ddots&&\vdots\\
0&0&\dots&\mathbf{D}^{2,M-1}&0\\
0&0&\dots&0&\mathbf{D}^{2,M}\\
\end{array}\right)
\end{equation}
which is applied to the global numerical solution $U$,
\begin{equation}\label{vectorUn}
U=\left(U_{0}^1,U_{1}^1,\dots,U_{N_1}^1,U_{0}^2,\dots,U_{N_{M-1}}^{M-1},U_{0}^M,\dots U_{N_M}^M \right)^T
\end{equation}
where $U_{i}^k\approx u(x_i^k)$, and because $x_{N_k}^k=x_0^{k+1}$, we take $U_{N_k}^k=U_{0}^{k+1}$ $k=1,\cdots, M-1$. The size of $\mathbf{D}^{2}$ is $\mathbf{N}\times\mathbf{N}$, where $\mathbf{N}=N_1+N_2+...+N_M+M$.

\subsection{Time integration}

\noindent With the definition of $\mathbf{D}^{2}$ and $U$ in Eqs. (\ref{matrixD2}) and (\ref{vectorUn}), respectively, Eq. \ref{fhn_ode_sys}, becomes an equation defined over the grid points
$$
\mathbf{x}=\left(x^1_0, \cdots, x^1_{N_1},x^2_0, \cdots, x^2_{N_2},\cdots x^M_0, \cdots, x^M_{N_M}\right),
$$
and is given by
\begin{equation}\label{fhn_ode_sys2}
\begin{array}{lcl}
\frac{dU}{dt}& = & \delta_u \mathbf{D}^{2}U + f(U,V)\\
\frac{dV}{dt}& = & \delta_v \mathbf{D}^{2}V + g(U,V).\\
\end{array}
\end{equation}

\noindent In order to solve such equations, we first describe the time integration part of our equations. For this we apply the operator splitting approach as done in \cite{qu99}, i.e. given an approximated solution at time $t^n$, $U^{k}_j(t^n)\approx u(x^k_j,t^{n})$, in order to advance a $\Delta t$ time
step, we
\begin{itemize}
    \item Solve initially the diffusion equation with a backward Euler method with time step $\Delta t/2$ and initial condition $U(t^n)$ and name this solution $U^{*}$, i.e.
\begin{equation}\label{diff_eq}
U^{*}=U(t^n)+\frac{\Delta t}{2}\delta_u \mathbf{D}^{2}U^{*} \qquad \textrm{and}
\qquad V^{*}=V(t^n)+\frac{\Delta t}{2}\delta_v \mathbf{D}^{2}V^{*}
\end{equation}
or
\begin{equation}\label{diff_eq2}
\left(I-\delta_u\frac{\Delta t}{2} \mathbf{D}^{2}\right)U^{*}=U(t^n) \qquad
\textrm{and} \qquad \left(I-\delta_v\frac{\Delta t}{2} \mathbf{D}^{2}\right)V^{*}=V(t^n)
\end{equation}
where $U(t^n)$ and $V(t^n)$ are given by,
\begin{equation}\label{vectorUntn}
U(t^n)=\left(U_{0}^1(t^n),U_{1}^1(t^n),\dots,U_{j}^k(t^n),\dots, U_{N_M}^M(t^n) \right)^T
\end{equation}
and
\begin{equation}\label{vectorVntn}
V(t^n)=\left(V_{0}^1(t^n),V_{1}^1(t^n),\dots,V_{j}^k(t^n),\dots, V_{N_M}^M(t^n) \right)^T.
\end{equation}

\item Solve the reaction part with a $\Delta t$ time step and initial condition $U^{*}, V^{*}$, by Forward Euler's method, i.e.
\begin{equation}
    \begin{array}{lcl}
    U^{\bigtriangleup}&=&U^{*}+\Delta tf(U^*,V^*)\\
    V^{\bigtriangleup}&=&V^{*}+\Delta tg(U^*,V^*).
    \end{array}
\end{equation}
\item Finally, we solve the diffusive part with time step $\Delta t/2$ and initial condition $U(t^{\bigtriangleup})$ and $V(t^{\bigtriangleup})$ to obtain $U(t^{n+1})$ and $V(t^{n+1})$
as
\begin{equation}\label{diff_eq2f}
\left(I-\delta_u\frac{\Delta t}{2} \mathbf{D}^{2}\right)U(t^{n+1})=U(t^{\bigtriangleup}) \quad
\textrm{and} \quad \left(I-\delta_v\frac{\Delta t}{2} \mathbf{D}^{2}\right)V(t^{n+1})=V(t^{\bigtriangleup}).
\end{equation}
\end{itemize}
Boundary and interface conditions are applied after an advance in time in the diffusion part takes place.

\noindent Therefore, we solve Eq. (\ref{fhn_ode_sys2}) with different time steps. In regions where there is no mesh refinement, we take a large time step $\Delta t$. However, when the solutions are in the refined mesh, we solve the equation for small time step $dt$.
At each $dt$ time step, we apply the boundary and interface conditions. We take $\Delta t=Kdt$ with $K$ a natural number. The reactive part of the equation was also solved with a second order Runge-Kutta method. However, by using simple Euler's integration, the results were practically the same. It is important to mention that in both cases, the refined and the coarse regions, the same time step $\Delta t$ was taken for the reactive part.
\subsubsection{Boundary and interface conditions}\label{subSec:BIC}

\noindent Additional to the initial and boundary conditions, we impose the interface conditions as the continuity of the approximation and its first derivative at the interface as
\begin{equation}\label{sc}
  U^k_{N_k}=U_{0}^{k+1}
 \end{equation}
 and
  \begin{equation} \label{sc2}
     \left.\frac{\partial I_Nu}{\partial x}\right|_{x^k_{N_k}}=
\left.\frac{\partial I_Nu}{\partial x}\right|_{x^{k+1}_{0}}
\end{equation}
for $k=1,...,M-1$. Here note that ensuring the continuity of the approximation only is not enough for stability as we will show later. Both continuities in the approximation and its derivative are essential for stability. Here we also note that the continuity is a strong condition. For the approach with  discontinuous Galerkin methods, continuity is only weakly enforced, resulting in discontinuous solutions across the subdomains.

\noindent We apply the boundary conditions simultaneously together with the interface conditions
\begin{equation}\label{bc}
  \left.\frac{\partial I_Nu}{\partial x}\right|_{x^1_{0}}=0
\qquad \textrm{and} \qquad
  \left.\frac{\partial I_Nu}{\partial x}\right|_{x^M_{N_M}}=0.
\end{equation}
The condition of Eq. (\ref{sc2}) at the interface between $k$th and $(k+1)$th subdomains is given by the following
$$
\sum_{j=0}^{N_k}D^{1,k}_{N_k,j}U_j^{k}=\sum_{j=0}^{N_{k+1}}D^{1,k+1}_{0,j}U_j^{k+1} \qquad \textrm{for} \qquad k=1, \cdots, M-1.
$$
If we let $x_*^k=x_{N_k}^k=x_{0}^{k+1}$, the condition in Eq. (\ref{sc}) leads us to define  $U_*^k:=U_{N_k}^k=U_{0}^{k+1}$. From here, we obtain a system of equations for the approximate solution at the interfaces, $U_*^1$, $U_*^2$, $...$, $U_*^{M-1}$. Once these interface solutions are determined, the continuity conditions at the interfaces can be written as

$$
\begin{array}{ll}
D^{1,k}_{N_k,0}U_*^{k-1}+D^{1,k}_{N_k,N_k}U_*^{k}-D^{1,k+1}_{0,0}U_*^k-D^{1,k+1}_{0,N_{k+1}}U_*^{k+1}&=\\
-\sum_{j=1}^{N_k-1}D^{1,k}_{N_k,j}U_j^{k}+\sum_{j=1}^{N_{k+1}-1}D^{1,k+1}_{0,j}U_j^{k+1}&
\end{array}
$$
for the $k$th and $(k+1)$th subdomains, and equivalently
$$
\begin{array}{ll}
D^{1,k}_{N_k,0}U_*^{k-1}+(D^{1,k}_{N_k,N_k}- D^{1,k+1}_{0,0})U_*^k-D^{1,k+1}_{0,N_{k+1}}U_*^{k+1}&=\\
-\sum_{j=1}^{N_k-1}D^{1,k}_{N_k,j}U_j^{k}+\sum_{j=1}^{N_{k+1}-1}D^{1,k+1}_{0,j}U_j^{k+1}&
\end{array}
$$
for $k=1,\cdots,M-1$.
Now, from the boundary condition Eqs. (\ref{bc}), we obtain the equations
$$
D^{1,1}_{0,0}U_0^{1}+D^{1,1}_{0,N_1}U_{*}^{1}=-\sum_{j=1}^{N_1-1}D^{1,1}_{0,j}U_j^{1}
$$
and
$$
D^{1,M}_{N_{M},0}U_{*}^{M-1}+D^{1,M}_{N_{M},N_{M}}U_{N_M}^M=-\sum_{j=1}^{N_{M}-1}D^{1,M}_{N_{M},j}U_j^M.
$$
Given the interior solutions, $U^k_j$, found, the boundary or interface solutions that satisfy the boundary and interface conditions can be obtained through the linear system $AU_{I} = b$, where
\begin{equation}\label{vectorBCCS}
 \scriptsize
  U_{I}=\left(
 \begin{array}{c}
  U_0^1\\
  U_*^1\\
  U_*^2\\
\vdots\\
U_*^{M-1}\\
U_{N_{M}}^{M}\\
  \end{array}
 \right).
\end{equation}
The coefficient matrix $A$ is
\begin{equation}\label{matrixBCSC}
\scriptsize
 A=\left(
\begin{array}{cccccc}
D^{1,1}_{0,0}&D^{1,1}_{0,N_1}&0&&\dots&0\\
D^{1,1}_{N_1,0}&d_1&-D^{1,2}_{0,N_2}&0&\dots&0\\
0&D^{1,2}_{N_2,0}&d_2&-D^{1,3}_{0,N_3}&\dots&0\\
\vdots&\ddots&\ddots&\ddots&&\vdots\\
0&0&D^{1,M-2}_{N_{M-2},0}&d_{M-2}&-D^{1,M-1}_{0,N_{M-1}}&0\\
0&&0&D^{1,M-1}_{N_{M-1},0}&d_{M-1}&-D^{1,M}_{0,N_{M}}\\
0&&\dots&0&-D^{1,M}_{N_{M},0}&-D^{1,M}_{N_{M},N_{M}}\\
\end{array}\right)
\end{equation}
where, $d_i=D^{1,i}_{N_i,N_i}-D^{1,i+1}_{0,0}$ and the interior solution vector, $b$ is
\begin{equation}\label{vectorbBCCS}
 \scriptsize
 b=\left(
 \begin{array}{c}
  -\sum_{j=1}^{N_1-1}D^{1,1}_{0,j}U_j^{1}\\
  -\sum_{j=1}^{N_1-1}D^{1,1}_{N_1,j}U_j^{1}+\sum_{j=1}^{N_{2}-1}D^{1,2}_{0,j}U_j^{2}\\
    -\sum_{j=1}^{N_2-1}D^{1,2}_{N_2,j}U_j^{2}+\sum_{j=1}^{N_{3}-1}D^{1,3}_{0,j}U_j^{3}\\
\vdots\\
  -\sum_{j=1}^{N_{M-1}-1}D^{1,M-1}_{N_{M-1},j}U_j^{M-1}+\sum_{j=1}^{N_c^{M}-1}D^{1,M}_{0,j}U_j^{M}\\
  \sum_{j=1}^{N_{M}-1}D^{1,M}_{N_{M},j}U_j^{M}\\
 \end{array}
 \right).
\end{equation}
The matrix $A$, is a tridiagonal matrix. It can be easily shown that $A \ne 0$. Note that the system can be solved efficiently with Thomas algorithm \cite{Zhang_bookchapter}.
\subsubsection{Algorithm to build $A$ and $b$}

\noindent Given the number of subdomains $M$ and the set of the number of quadrature points in each subdomain $\mathbf{N}=(N_1,N_2,..,N_M)$, $A$ and $b$ are computed algorithmically as follows: $A \in \mathbb{R}^{(M+1)\times(M+1)}$ is a tridiagonal matrix such that $A(1,1)=D^{1,1}_{0,0}$, $A(1,2)=D^{1,1}_{0,N_1}$, $A(M+1,M)=-D^{1,M}_{N_M,0}$ and
$A(M+1,M+1)=-D^{1,M}_{N_M,N_M}$. Now, for $s=2$ to $M$ we have that
\begin{equation}
\begin{array}{lll}
A(s,s-1)&=&D^{1,s-1}_{N_{s-1},0}\\ A(s,s)&=&D^{1,s-1}_{N_{s-1},N_{s-1}}-D^{1,s}_{0,0}\\
A(s,s+1)&=&-D^{1,s}_{0,N_s}.
\end{array}
\end{equation}
In the same way, $b \in \mathbb{R}^{(M+1)}$ is computed as: $b(1)=-\sum_{j=1}^{N_1-1}D^{1,1}_{0,j}U_j^{1}$, $b(M+1)=\sum_{j=1}^{N_{M}-1}D^{1,M}_{N_{M},j}U_j^{M}$ and for $s=2$ to $M$
\begin{equation}
\begin{array}{lll}
b(s)&=&-\sum_{j=1}^{N_{s-1}-1}D^{1,s-1}_{N_{s-1},j}U_j^{s-1}+\sum_{j=1}^{N_{s}-1}D^{1,s}_{0,j}U_j^{s}.\\ \end{array}
\end{equation}
\subsection{Numerical stability of the method}\label{Sec:Stability}

\noindent In this section we discuss the stability of the method.  Particularly, we focus on the solution of the diffusive part of the equation. In this case, the operator splitting scheme is not required and therefore, we focus on  scheme (\ref{diff_eq}) with boundary and interface conditions, i.e.
\begin{equation}\label{diff_eq_diff}
\left(I-\delta_u\Delta t \mathbf{D}^{2}\right)U(t^{n+1})=U(t^n) \quad
\textrm{and} \quad \left(I-\delta_v\Delta t \mathbf{D}^{2}\right)V(t^{n+1})=V(t^n).
\end{equation}
Scheme (\ref{diff_eq_diff}) provides directly, the value of $U(t^{n+1})$ and $V(t^{n+1})$. In this situation, a backward Euler method is implemented for a time step, followed by the implementation of the boundary and interface conditions. From here, we can think on the recurrence iteration,
\begin{equation}\label{method_discrete}
U^{n+1}=S_2U^*=S_2(S_1U^n)=(S_2S_1)U^n
\end{equation}
where $S_1=(I-\delta_u\Delta t\mathbf{D}^{2})^{-1}$, $I$ is the identity matrix, $\mathbf{D}^{2}$ is given in expression (\ref{matrixD2}) and $U^n$ is as given in (\ref{vectorUn}). $S_2$ is an operator that applies the boundary and interface conditions to $U^*=S_1U^n$. A similar equation is taken for the $V$ variable.

\noindent The $S_2$ operator can be obtained from the linear system $AU_I=b$ given by expressions (\ref{matrixBCSC}) and (\ref{vectorbBCCS}). First, observe that $U_I$ in expression (\ref{vectorBCCS}), is $U^{n+1}$ at the boundary and interface points. As we look for an expression of the form $U^{n+1}=S_2U^*$ in which all the points
of the mesh are considered, we can extend $A$ and $b$ to all points in the mesh. In fact, $b$, extended to all the domain points is given by $\hat{b}=BU^*$, where

\begin{equation}\label{Bbmatrix}
B=\left(
\begin{array}{cccccccccc}
0&-\mathbf{D}^{1,1}_{0,1}&\dots&-\mathbf{D}^{1,1}_{0,N_1-1}&0&0&0&\dots&0&0\\
0&1&0&0&0&0&\dots&0&0&0\\
0&0&\ddots&0&0&0&\dots&0&0&0\\
0&0&0&1&0&0&0&\dots&0&0\\
0&-\mathbf{D}^{1,1}_{N_1,1}&\dots&-\mathbf{D}^{1,1}_{N_1,N_1-1} &0&0&\mathbf{D}^{1,2}_{0,1} &\dots&   \mathbf{D}^{1,2}_{0,N_2-1}&0\\
0&-\mathbf{D}^{1,1}_{N_1,1}&\dots&-\mathbf{D}^{1,1}_{N_1,N_1-1} &0&0&\mathbf{D}^{1,2}_{0,1} &\dots&   \mathbf{D}^{1,2}_{0,N_2-1}&0\\
0&0&\dots&0&0&0&1&0&\dots&0\\
0&0&\dots&0&0&0&0&\ddots&0&0 \\
0&0&\dots&0&0&0&\dots&0&1&0 \\
0&0&\dots&0&0&0&\mathbf{D}^{1,2}_{N_2,1}&\dots& \mathbf{D}^{1,2}_{N_2,N_2-1}&0\\
\end{array}\right)
\end{equation}
and
\begin{equation}\label{ustar}
U^*=\left(U^*(x_0^1),U^*(x_{1}^1),\dots, U^*(x_{N_1}^1),U^*(x_0^2), \dots, U^*(x_{N_2}^{2}) \right)^T
\end{equation}
in the case we have two subdomains. Observe that for $U^*$, its values at the interface points $x_{N_1}^1$ and $x_0^2$ are not necessarily equal. It follows that $\hat{b}$ for two subdomains is given by
$$
  \hat{b}=\left(
 \begin{array}{c}
-\sum_{j=1}^{N_c^1-1}  \mathbf{D}^{1,1}_{0,j} U^*(x_j^{1})\\
  U^*(x_1^1)\\
  U^*(x_2^1)\\
\vdots\\
U^*(x_{N_1-1}^1)\\
-\sum_{j=1}^{N_1-1}  \mathbf{D}^{1,1}_{N_1,j}U(x_j^{1})+\sum_{j=1}^{N_2-1}\mathbf{D}^{1,2}_{0,j}U(x_j^{2})\\
\mathbf{D}^{1,1}_{N_1,j}U(x_j^{1})+\sum_{j=1}^{N_c^{2}-1}\mathbf{D}^{1,2}_{0,j}U(x_j^{2})\\
  U^*(x_1^2)\\
  U^*(x_2^2)\\
\vdots\\
U^*(x_{N_2-1}^2)\\
  \sum_{j=1}^{N_{2}-1}\mathbf{D}^{1,2}_{N_2,j}U(x_j^{2})\\
  \end{array}
 \right).
$$

\noindent On the other hand, $A$, which operates over the boundary and the interface points, can be extended to a new operator $\hat{A}$ which acts over all the mesh points and such that leaves invariant all the values located off the boundary and interface points. This
new matrix, is given by
\begin{equation}\label{matrixAhat}
\hat{A}=\left(
\begin{array}{ccccccccccc}
D_0&0&0&0&0&-\frac{(-1)^{N_1}}{2}&0&0&\dots&0&0\\
0  &        1              &0&0&0&0&0&0&\dots&0&0\\
0  &        0              &1&0&0&0&0&0&\dots&0&0\\
\vdots &\vdots      &\vdots&\ddots&0&0&0&0&\dots&0&0\\
0&\dots&0&0&1&0&0&0&\dots&0&0\\
\frac{(-1)^{N_1}}{2}&0&\dots&0&0&D_1&0&0&\dots&0&-\frac{(-1)^{N_2}}{2}\\
\frac{(-1)^{N_1}}{2}&0&\dots&0&0&0&D_1&0&\dots&0&-\frac{(-1)^{N_2}}{2}\\
0&0&0&\dots&0&0&0&1&0&0&0\\
0&0&0&\dots&0&0&0&0&\ddots&0&0\\
0&0&0&\dots&0&0&0&0&0&1&0\\
0&0&0&\dots&0&0&\frac{(-1)^{N_{2}}}{2}&0&0&0&D_2\\
\end{array}\right),
\end{equation}
for two subdomains and where $D_0=\frac{2(N_1)^2+1}{6}$, $D_1=-\frac{(N_{1})^2+(N_{2})^2+1}{3}$ and $D_{2}=-\frac{2(N_{2})^2+1}{6}$.
\noindent Now, it follows that $U^{n+1}$ satisfies $\hat{A}U^{n+1}=\hat{b}$, with
$$
U^{n+1}=\left(U^{n+1}(x_0^1),U^{n+1}(x_{1}^1),\dots, U^{n+1}(x_{N_1}^1),U^{n+1}(x_0^2), \dots, U^{n+1}(x_{N_2}^{2}) \right)^T, $$
and
$$U^{n+1}(x_{N_1}^1)=U^{n+1}(x_0^2).$$
From the previous reasoning, we obtain that the system $AU^{n+1}=b$, defined over the boundary and connecting points can be extended to the system $\hat{A}U^{n+1}=\hat{b}$ defined over the whole domain. Clearly, $U^{n+1}=\hat{A}^{-1}\hat{b}=\hat{A}^{-1}BU^*$, and therefore $S_2=\hat{A}^{-1}B$.

\noindent Observe that $B$ is an idempotent matrix. It follows that $B$ has eigenvalues both $1$ and $0$.
The eigenvalues have  equal geometric and algebraic dimension. Also, it is clear that $\hat{A}^{-1}B$ is an idempotent matrix as well.

\noindent
\begin{lemma}\label{lemma_eigD2}
The eigenvalues $\lambda_i$, $i=1,...,n$ of matrix $S_1=(I-\delta_u\Delta t\mathbf{D}^{2})^{-1}$ with homogeneous Dirichlet boundary conditions are real and belong to the interval $(0,1]$.
\end{lemma}
\begin{proof}
Initially, matrix $\mathbf{D}^{2}$ has four zero eigenvalues and the rest are real and negative. Therefore, $(I-\delta_u\Delta t D^2)$ has four eigenvalues 1 and the rest are real and greater than $1$. Therefore, it follows that four eigenvalues of $S_1$ are $1$ and the rest satisfy $\lambda_i \in (0,1)$.
\end{proof}

\begin{lemma}\label{lemma_eigS}
The eigenvalues of matrix $S=S_2S_1=\hat{A}^{-1}B(I-\delta_u\Delta t\mathbf{D}^{2})^{-1}$, such that $\mathbf{D}^{2}$ has homogeneous Dirichlet boundary conditions are non-negative and less than $1$.
\end{lemma}
\begin{proof}
The $S_1$ matrix with two subdomains is given by
\begin{equation}\label{matrixS_1}
S_1=\left(
\begin{array}{llllllllll}
1 & 0 &\cdots  &0&0 & 0&0&\cdots&0&0\\
0&\mathbf{D}^{2,1}_{1,1}&\cdots&\mathbf{D}^{2,1}_{1,N_1-1}&0&0&0&\cdots&0&0\\
\vdots&\vdots&\ddots&\vdots&&&&&\vdots&\vdots\\
0&\mathbf{D}^{2,1}_{N_1-1,1}&\cdots&\mathbf{D}^{2,1}_{N_1-1,N_1-1}&0&0&0&\cdots&0&0\\
0&0&\cdots&0&1&0&0&\cdots&0&0\\
0&0&\cdots&0&0&1&0&\cdots&0&0\\
0&0&\cdots&0&0&0&\mathbf{D}^{2,2}_{1,1}&\cdots&\mathbf{D}^{2,2}_{1,N_2-1}&0\\
\vdots&\vdots&&&&&\vdots&\ddots&\vdots&\vdots\\
0&0&\cdots&0&0&0&\mathbf{D}^{2,2}_{N_2-1,1}&\cdots&\mathbf{D}^{2,2}_{N_2-1,N_2-1}&0\\
0&0&\cdots&0&0&0&0&\cdots&0&1\\
\end{array}\right)
\end{equation}
Clearly, each of the two sub-blocks have determinant different from zero and have real and non-negative eigenvalues less than $1$ . Now, the product $S_2S_1$ has the structure
\begin{equation}\label{matrixS2S1}
S_2S_1=\left(
\begin{array}{llllllllll}
0 & K_1^1 &\cdots  &K_{N_1-1}^1&0 & 0&0&\cdots&0&0\\
0&\mathbf{D}^{2,1}_{1,1}&\cdots&\mathbf{D}^{2,1}_{1,N_1-1}&0&0&0&\cdots&0&0\\
\vdots&\vdots&\ddots&\vdots&&&&&\vdots&\vdots\\
0&\mathbf{D}^{2,1}_{N_1-1,1}&\cdots&\mathbf{D}^{2,1}_{N_1-1,N_1-1}&0&0&0&\cdots&0&0\\
0&K_1^2&\cdots&K_{N_1-1}^2&0&0&0&\cdots&0&0\\
0&0&\cdots&0&0&0&K_{1}^3&\cdots&K_{N_2-1}^3&0\\
0&0&\cdots&0&0&0&\mathbf{D}^{2,2}_{1,1}&\cdots&\mathbf{D}^{2,2}_{1,N_2-1}&0\\
\vdots&\vdots&&&&&\vdots&\ddots&\vdots&\vdots\\
0&0&\cdots&0&0&0&\mathbf{D}^{2,2}_{N_2-1,1}&\cdots&\mathbf{D}^{2,2}_{N_2-1,N_2-1}&0\\
0&0&\cdots&0&0&0&K_{1}^4&\cdots&K_{N_2-1}^4&0\\
\end{array}\right)
\end{equation}
From here, clearly follows that has four eigenvalues zero and the rest are real, positive and less than one.
\end{proof}
\noindent We end this section with our main result
\begin{theorem}[Stability of the Method]
The numerical method given by equation \ref{diff_eq_diff}, with no flux boundary conditions and interface conditions, is stable.
\end{theorem}

\subsection{Switching from coarse to refined meshes}\label{Sec:switch}

\noindent For the present manuscript, we have considered two mesh types. Each subdomain can have a coarse mesh with $N_c$ points or a fine mesh with $N_f$ collocation points. The coarse mesh is used at locations where the function has no representative changes in its value and the refined mesh is taken at subdomains with large derivative. The switching procedure from coarse to fine grids is as follows:  Initially we define a coarse mesh, with $N_c$ points in each subdomain, where the initial condition is defined. Then, we evaluate the first derivative of the variable that experiences abrupt changes. At subdomains where the absolute value of the first derivative is larger than a given value $\eta$, we proceed to refine the subdomain with $N_f$ points. For this, let the $kth$ mesh be the one to be refined. Let $U^{k,N_c}$ be the vector with $N_c$ components, which is the approximate solution at the $kth$ subdomain evaluated with the coarse mesh. Then, we find the corresponding frequency modes from the node values (See \cite{Shen_book})
\begin{equation}\label{Eq.modesak}
    a_k=\frac{1}{\tilde{c}_kN_c}\sum_{j=0}^{N_c}\frac{1}{\tilde{c}_j}U^{k,Nc}_jT_k(x_j), \qquad 0\leq k \leq N_c
\end{equation}
and $\tilde{c}_0=\tilde{c}_N=2$ and $\tilde{c}_k=1$, otherwise. Once the $a_k$ modes are found, we use the first relation in Eq. (\ref{expand}) evaluated at the refined mesh, thus giving $U^{k,N_f}$, the solution at the $kth$ evaluated with the refined mesh.

\noindent We end this section with some remarks:
\begin{enumerate}
    \item We follow the same procedure in order to change a refined mesh to a coarse mesh.
    \item The value of $\eta$ is obtained experimentally and has to be regulated depending on the equation to be solved. A large value of $\eta$ will not detect the fast transitions and a small enough values will lead to refine subdomains with small changes.
    \item In order to assure a correct propagation of the pulse, we also enforce all the immediate neighbors to be refined with the same $N_f$. In the case of the two-dimensional simulations, we also refine the eight subdomains around the refined one. At each time step, we calculate the first derivative in each subdomain and the value $\eta$ is the responsible to decide if mesh stays coarse or refined. Observe that refining the neighbors in both sides is needed even the traveling wave is one-directional. The main reason is that some RD equations, such as the Gray-Scott model, have the wave splitting phenomenon \cite{pet94}, in which suddenly a pulse splits into two directions. One traveling from left to right, and the other traveling from right to left.
\end{enumerate}


\section{Numerical results}\label{Sec_numres}

\noindent The purpose of the numerical simulations is to show the usefulness of our developed method. We present solutions for the one
and two-dimensional cases. For this work, all the computations were done in a Laptop HP Quad-Core Intel Core i7-6700HQ CPU 2.60GHz, with elementary OS 0.4.1 Loki (64-bit) in a serial code. Also, for the present simulations we have used the library chebdif.m from \cite{Baltensperger2006}.

\noindent For our simulations, two approaches are considered to solve the equations. The first method (CMD1), uses a fixed grid given by $M$ subdomains and $N$ collocation points in each subdomain in which the subdomains have only one point in common. The second, which is our proposed method (AMCMD1), uses $M$ subdomains, and for each subdomain, we can switch grids with $N_c$ points for the coarse grid and $N_f$ points for the refined grid.

\subsection{One dimensional problem}

\noindent  In order to verify the efficiency of the method, we consider three numerical experiments. The first is the convergence of CMD1. This experiment is done over a small space interval and short integration time. Convergence for AMCMD1 follows directly as it is a special case of CMD1. In the second experiment we verify the usefulness of AMCMD1. This is done by taking a much larger interval and use the speed of the pulse as a parameter to explore the trade between the computing time and
accuracy. The speed of the wave provides a robust measurement of the accuracy.
The last experiment shows a simulation of an even larger domain and present estimations about the improvement
gained with AMCMD1.

\noindent Convergence of CMD1 is provided by the use of the $l_\infty$ error between the exact and approximate solutions. In our case,
exact solutions are not available in an analytical form and therefore, our exact solution $u_e$ is given by a
solution obtained with CMD1 with a sufficiently large number of points such that other solutions with less
points than $u_e$ has an error less than $10^{-10}$. $l_\infty$ error between an approximated solution $u_a$ and $u_e$ is
 evaluated at the mesh where $u_a$ is defined. To find $u_e$ evaluated at the same mesh as $u_a$, we use (\ref{lobato}), (\ref{expand}) and interpolation.

\subsubsection{Fitzhugh-Nagumo pulses}

\noindent For the Fitzhugh-Nagumo equations (Eqs. \ref{FHNeq}), the physical parameters are $a=0.3$, $b=0.01$ and
$\epsilon=0.005$. The initial condition is given by
\begin{equation}\label{ic_fhn1d}
 u_0(x)=\frac{1}{\left(1+e^{4(x-r_1)}\right)^2}-\frac{1}{\left(1+e^{4(x-r_2)}\right)^2}
\end{equation}
and
\begin{equation}\label{ic_fhn1dv}
v_0(x)=0
\end{equation}
where $r_1=5$ and $r_2=2$. This initial condition evolves into a propagating pulse that
moves from left to right along the spatial domain.\\

\noindent \textbf{Convergence analysis for CMD1}

\noindent The convergence analysis is done for $x \in [0,10]$ and final time $T=0.5$. In order to observe convergence,
we compare numerical solutions to the "exact`` solution obtained with CMD1 with numerical parameters ($M=80$, $N=20$, $\Delta_t=2.5\times10(-7)$), which gives $N_p=1521$ discretization points. Table \ref{Table:convergenceFHN1}, clearly shows that by choosing ($M=70$, $N=20$), we obtain an error of the order of $10^{-11}$. \\
\begin{table}
 \begin{center}
 \begin{tabular}{||l | l| l| l ||}
 \hline
 $M$ & $N_p$ & $\Delta t$ & Error\\ [0.5ex]
 \hline\hline
 5  & 96  & $1.000(-5)$& $2.467(-3)$ \\
 \hline
 40 & 761 &$1.000(-6)$ & $1.585(-4)$ \\
 \hline
 50 & 951 & $5.000(-7)$& $5.285(-5)$ \\
 \hline
 60 & 1141 &$2.500(-7)$ & $1.697(-10)$\\
 \hline
 70 & 1331 &$2.500(-7)$ & $8.753(-11)$\\
 \hline
\end{tabular}
\caption{Numerical convergence for CMD1 to solve the one dimensional FHN equations (Eqs. \ref{FHNeq}). $N_p$ is the total number of discretization
points. $x \in [0,10]$. For all the computations $N=20$. }\label{Table:convergenceFHN1}
\end{center}
\end{table}

\noindent \textbf{Speed of the wave}

\noindent A second useful measure of the method's accuracy is the speed of the wave. In order to calculate it we solved Eq. (\ref{FHNeq}) for $x\in[0,200]$ and a final time $T=18$. In this case, convergence of the solution is more difficult to observe than in the previous experiment  due to the length of the domain and larger final time. The initial condition is as in Eq. (\ref{ic_fhn1d}) with $r_1=20$ and $r_2=10$.
In Table \ref{Table:convergencespeedCMD1FHN1}, we show that the speed of the wave is getting close to $9.0753$. With
CMD1, the solution with the least number of points and $\Delta t=1e(-4)$ was obtained with $M=600$ and $N=3$.
By taking  $M=500$ and $N=3$, the value of $\Delta t=1e(-5)$ is required for stability purposes.
\begin{table}
 \begin{center}
 \begin{tabular}{||l | l| l| l | l ||}
 \hline
 $M$ & $N_p$ & $\Delta t$ & speed & CPU time\\ [0.5ex]
 \hline\hline
500  & 1001  & $1.0000(-5)$& $9.1653$& 2,805\\
 \hline
 600  & 1201  & $1.0000(-4)$& $9.1873$& 168\\
 \hline
 1200 & 2401 &$1.0000(-4)$ & $9.1493$&  648 \\
 \hline
 1800 & 3601 & $1.0000(-4)$& $9.0980$& 1,237 \\
 \hline
 2300 & 4601 &$1.0000(-4)$ & $9.0776$ & 1,386\\
 \hline
 2500 & 5001 &$1.0000(-4)$ & $9.0724$& 1,757 \\
 \hline
  2800 & 5601 &$5.0000(-5)$ & $9.0753$ & 3,879\\
 \hline
\end{tabular}
\caption{Numerical convergence of the speed for CMD1 to solve the one dimensional FHN equations (Eqs. \ref{FHNeq}). Equation solved
for $t \in [0,18]$, $x\in[0,200]$ and $N=3$. CPU time in seconds.}\label{Table:convergencespeedCMD1FHN1}
\end{center}
\end{table}

\noindent Now, our aim is to come with AMCMD1 solutions that use considerably less number of points and computing time without losing
numerical accuracy. We run simulations with $M=60$, $N_c=3$, $N_f=20$, $\Delta t=0.0001$ and varied the value of $M$ in order
to have larger values of $\Delta t$. From Table \ref{Table:speedAMCMD1FHN1}, it is clear that by using either Euler's method or second order Runge-Kutta, the solutions have practically the same accuracy. Moreover, in most of the simulation only $6$ subdomains have $N_f$ points and the rest used $N_c$ points. This leads to much less number of points ($Np=462$) yet as accurate as the one obtained with CMD1 and $N_p=3601$ points. This clearly also leads to a faster solution eight times faster than the solution with CMD1.

\begin{table}
 \begin{center}
 \begin{tabular}{||l | l| l| l | l ||}
 \hline
 $M$ & $dt$ & Method & speed & CPU time\\ [0.5ex]
 \hline\hline
 1 & 1.0000(-4)  & RK2 & $9.0452$& 153.55\\
 \hline
 1  & 1.0000(-4)  & Euler & $9.0407$& 154.82 \\
  \hline
 10  & 5.0000(-5)  & RK2 & $9.0514$ & 89.97 \\
  \hline
 10  & 5.0000(-5)  & Euler & $9.0157$ &89.92 \\
 \hline
 20  & 5.0000(-5)  & RK2 & $-$ & $-$ \\
 \hline
 20  & 5.0000(-5)  & Euler & $8.9819$ & 71.78 \\
 \hline
 100  & 5.0000(-5)  & RK2 & $9.0444$ & 60.81 \\
 \hline
 100  & 5.0000(-5)  & Euler & $-$ & $-$ \\
 \hline
\end{tabular}
\caption{Performance of AMCMD1 to solve the one dimensional FHN equations (Eqs. \ref{FHNeq}) for
different values of $M$ and numerical integrator for the reactive part. $N_c=3$, $N_f=20$, $dt=0.0001$ and $M=60$.
$N_p=462$ is the total number of points used in order to solve the propagating pulse. CPU time in seconds.} \label{Table:speedAMCMD1FHN1}
\end{center}
\end{table}

\begin{table}
 \begin{center}
 \begin{tabular}{||l | l| l | l ||}
 \hline
 Method & $N_p$  &  speed & CPU time \\ [0.5ex]
 \hline\hline
 CMD1 & 12001    &$9.1858$& 45200\\
 \hline
 AMCMD1 & 486    &$8.9464$& 771\\
 \hline
\end{tabular}
\caption{Comparison of performance between CMD1 and AMCMD1 to solve the one dimensional FHN equations (Eqs. \ref{FHNeq}).
methods for large domains. CPU time in seconds.}\label{Table:AMCMD1_CMD1_FHN1_ld}
\end{center}
\end{table}

\noindent From the results, it follows that AMCMD1 provides wih good approximations with less points and less computational effort than CMD1.
This advantage becomes more evident when a much larger domain is taken. For this, we take the domain $x \in [0,2000]$, final
time $T=200$ and initial condition taken as in the previous study. The numerical parameters for AMCMD1 are $M=100$, $N_c=3$,
$N_f=60$, $\Delta t=0.0001$, $\eta=0.0001$ and $M=10$ and the time integrator is the second order Runge-Kutta. For CMD1, we took
$M=6000$ subdomains and $N=3$ points per subdomain. The result of this comparison is given in Table \ref{Table:AMCMD1_CMD1_FHN1_ld}. For
AMCMD1 only $5$ subdomains require $N_f$ points, whereas the rest use a coarse refinement. This implies that in average AMCMD1 requires
$486$ points to discretize the whole domain. Clearly, from the comparison, AMCMD1 is $58$ times faster than CMD to obtain the solution for the same integration time. Also, AMCMD1 only requires twenty-four times less points than CMD1.

\begin{figure}[h]
	\centering
	\includegraphics[width=13cm, height=4cm]{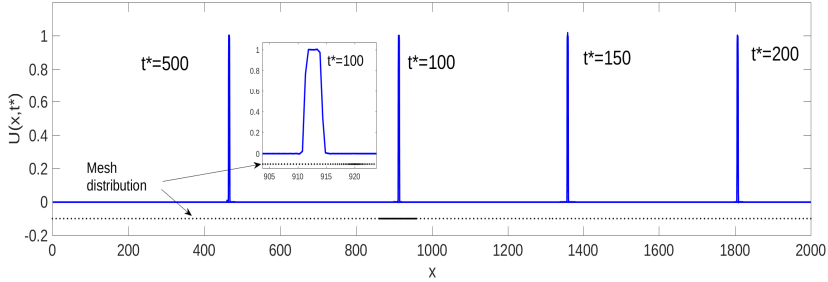}
	\caption{Numerical solutions for the one dimensional FHN equations (Eqs. \ref{FHNeq}) for different integration times. A zoom to the
	propagating pulse and mesh distribution for time $t^*=100$ is also shown.}
	\label{Fig:pulses_FHN_1d_AMCMD1}
\end{figure}

\noindent In Figure \ref{Fig:pulses_FHN_1d_AMCMD1} we show the numerical solution for the one-dimensional pulse at $t^*=50, 100, 150$ and $200$.
Within the figure, we show the zoomed image of the pulse for $t^*=100$. It is important to stress out that for the current simulation, the width of the pulse is approximately 3 space units, whereas the distance between two consecutive points in the coarse mesh is approximately 10 space units. This means that pulses cannot propagate on the coarse mesh. On the other hand the maximum distance between two consecutive points in the refined mesh is approximately 0.5, whereas the minimum distance is approximately $=0.014$.

\subsubsection{Gray-Scott pulses}

\noindent  The Gray-Scott model (Eq. \ref{GSeq}) is solved with parameters $\tau_r=315$, $k_2=1/40$
$b_0=1/15$ and $\delta=7$. A numerical convergence study is done over the interval $x \in [0,25]$ and final time $T=30$. The initial condition is given by
\begin{equation}\label{IC_GS_1d}
u_0(x)=0.936955 \qquad \textrm{and} \qquad v_0(x)=\frac{0.1}{\left(1+\exp{\frac{\sqrt{6\rho} x-5\rho k}{6}}\right)}+0.014615
\end{equation}
where $\rho=10000$ and $k=0.1$. In this case, the values $u_0=0.936955$ and $v_0= 0.014615$ are the coordinates of the asymptotically
stable fixed point of the system with no diffusion. Therefore, the initial condition considered is just a spatial perturbation of this fixed point, in order to obtain a
solitary propagating pulse.\\

\noindent \textbf{Convergence analysis for CMD1}\\
\noindent Here, we took the exact solution as the one obtained with CMD1 and ($M=220,N_c=10$ and $\Delta t=2.5e(-6)$). Table \ref{Table:convergenceGS1}, show the convergence of the method. As we are tracking the position of a solitary pulse,
the computed error for Table \ref{Table:convergenceGS1} becomes difficult to reduce as the computation becomes
costly.

\begin{table}
 \begin{center}
 \begin{tabular}{||l | l| l| l ||}
 \hline
 $M$ & $N_p$ & $\Delta t$ & Error\\ [0.5ex]
 \hline\hline
 40  & $361$  & $2.5000(-5)$& $2.2021(-4)$ \\
 \hline
 60 & $541$ &$2.5000(-5)$ & $6.0457(-5)$ \\
 \hline
 110 & $991$ & $1.0000(-5)$& $1.5517(-6)$ \\
 \hline
 180 & $1621$ &$2.5000(-6)$ & $3.6260(-7)$\\
 \hline
 200 & $1801$ & $2.5000(-6)$ & $4.6185(-8)$\\
 \hline
\end{tabular}
\caption{Numerical convergence for CMD1 method to solve the Gray-Scott equations (Eqs. \ref{GSeq}) ($N_c=10$), for
the domain $x \in [0,25]$ and $T=30$.} \label{Table:convergenceGS1}
\end{center}
\end{table}

\noindent Following the previous example, we focus on the numerical approximation of the speed of the wave as in \cite{os09}
as it provides a better appreciation of the performance of the method. Table \ref{Table:convergencespeedCMD1GS1} shows the results for CMD1 for a relatively large size of the domain $x\in[0,2000]$, final time $T=4000$ and the initial condition Eq. (\ref{IC_GS_1d}). From the table, it is clear that CMD1 can be computed fast with relatively few points ($N_p=601$). By repeating the same experiment for AMCMD1, we note that one of the best choices is given by the solution with $M=20$, $N_c=3$, $N_f=10$, $dt=0.005$, $\eta=0.0001$ and $M=200$.
Under this situation the speed of the wave becomes $c=0.2185021$. The computation requires $N_p=76$ points as the pulse is contained in only $5$ subdomains in average at each time instant during the computation. However, the computation time takes approximately $102$ seconds, which
becomes slower compared to CMD1 although the number of collocation points has been reduced considerably.

\begin{table}
 \begin{center}
 \begin{tabular}{||l | l| l| l | l ||}
 \hline
 $M$ & $N_p$ & $\Delta t$ & speed & CPU time\\ [0.5ex]
 \hline\hline
300  & 601  & $5.0000e(-1)$& $0.230076 $& 16\\
 \hline
 400  & 801  & $5.0000e(-1)$& $0.227390$ & 22\\
 \hline
  500 & 1001 &$5.0000e(-1)$ & $0.225324$&  26.7 \\
 \hline
 600 & 1201 & $1.0000e(-1)$& $0.226311$&160  \\
 \hline
700 & 1401 &$1.0000e(-1)$ & $0.225450$ & 197\\
 \hline
 700 & 1401 &$1.2500e(-3)$ & $0.226024$& 14,373 \\
 \hline
  800 & 1601 &$1.0000e(-1)$ & $0.224876$& 215 \\
 \hline
  800 & 1601 &$1.2500e(-3)$ & $0.225440$& 14,109 \\
 \hline
\end{tabular}
\caption{Numerical convergence of the speed for CMD1 to solve the one dimensional GS equations (Eq. \ref{GSeq}). $x \in [0,2000]$, $T=4000$ and $N_c=3$. CPU time in seconds.} \label{Table:convergencespeedCMD1GS1}
\end{center}
\end{table}

\noindent However, as observed in the case of the FNH equations, the real advantage of AMCMD1 comes when large simulation domains compared to the
width of the pulse is taken. For example, if we extend our simulation to the domain $x\in[0,20000]$ and $T=64000$ it is possible to use AMCMD1
in a much more advantageous way compared to CMD1. In Figure \ref{Fig:pulses_GS_1d_AMCMD1} we show a pulse propagation obtained with the GS equations (Eq. \ref{GSeq}) in such domain.

\begin{figure}[h]
	\centering
	\includegraphics[width=13cm, height=4cm]{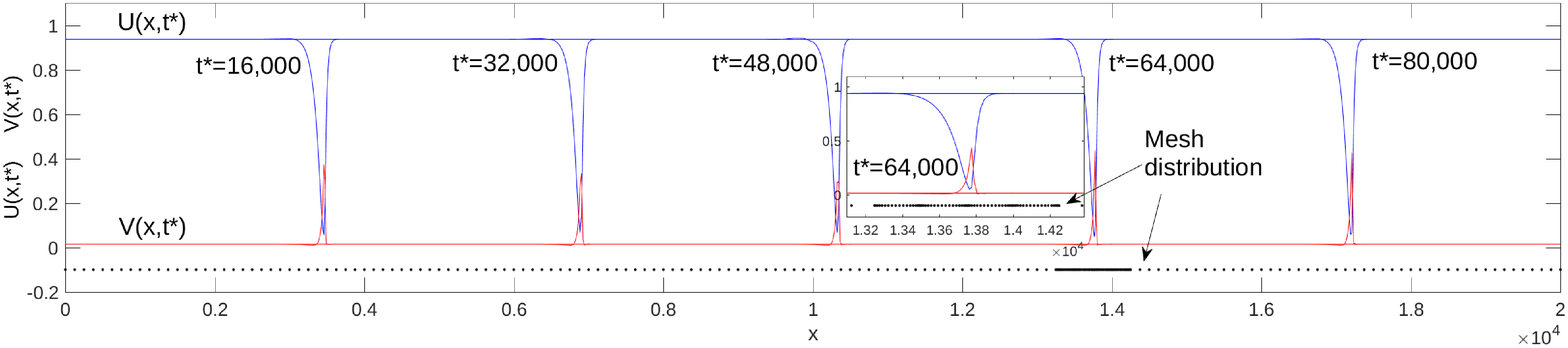}
	\caption{Numerical solutions for the one dimensional Gray-Scott equations (Eqs. \ref{GSeq}) for different integration times. A zoom to the
	propagating pulse and mesh distribution for time $t^*=64,000$ is also shown.}
	\label{Fig:pulses_GS_1d_AMCMD1}
\end{figure}

\subsubsection{CPU time versus the size of the domain}

\noindent From the previous computations, clearly AMCMD1 becomes a better option for relatively larger domains. In
order to test the efficiency of the method, we calculated the speed of the wave for different size domains such that the
speed of the wave is $c\approx 0.22\pm 0.01$. In Fig \ref{Fig:CPUandNpvsL} the CPU time and the total number of discretization
points versus the length of the domain $L$, is shown. From the figure, it follows that for small values of $L$, CMD1 performs better than AMCMD1,  even the number of discretization points is similar. However, AMCMD1 takes over when $L$ becomes larger. From the figure, for a fixed but large $L$ the CPU time with AMCMD1, halves the one obtained with CMD1 but with a much smaller amount of points. The advantage of AMCMD1 on this one dimensional problem relies on having only one single propagating pulse. AMCMD1 will slow down when more than one pulse are taken into account.

\begin{figure}[h]
	\centering
	\includegraphics[width=13cm, height=5cm]{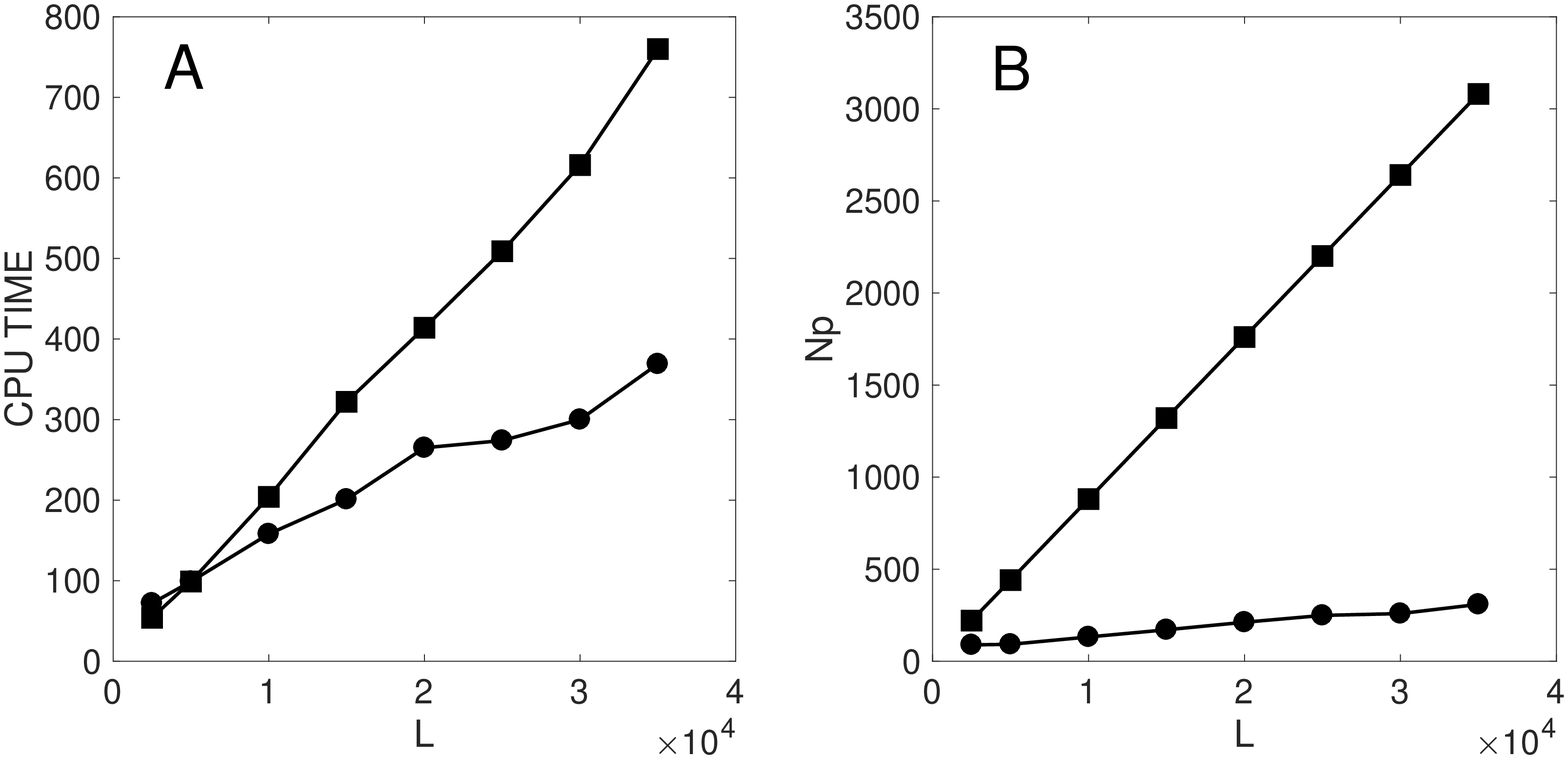}
	\caption{Performance of CMD1 and AMCMD1 for the Gray-Scott equations (Eqs. \ref{GSeq}) . (A) CPU time versus $L$ and (B)
	Number of discretization points $N_p$ versus $L$. CMD ($\blacksquare$) and AMCMD1 ($\bullet$).}
	\label{Fig:CPUandNpvsL}
\end{figure}

\subsection{Two-dimensional problem}

\noindent Wave propagation and pattern formation are the two more explored dynamics, numerically and analytically for RD systems. In the present section we approach both problems and observe how the AMCMD1 method performs in each case.

\noindent One of the main differences between the one and two-dimensional problems is the implementation of interface and boundary conditions. In order to implement such conditions for 2D, we follow the scheme shown in Fig. \ref{Fig:twodomainmeshBC}. In the figure, we show the case of two subdomains with three and four collocation points in each dimension, and where conditions are to be applied in the $x$ direction. The original collocation points are shown in black ($\bullet$) and ($\times$), for rectangular subdomains $1$ and $2$, respectively. In general, to apply the interface and boundary conditions for the first and last rows follow directly the one dimensional case (Section \ref{subSec:BIC}). In order to apply the interface and boundary conditions for the second row of the first subdomain, we need to find the values of the approximate solution evaluated at the points denoted by (\textcolor{blue}{$\diamond$}) in Subdomain $2$. Each value located at (\textcolor{blue}{$\diamond$}), is found by interpolating using the corresponding known values given at ($\times$) in the same column. This is done by finding the corresponding modes given by Eq. (\ref{Eq.modesak}), and then evaluating the approximation using Eq. (\ref{expand}) at the location (\textcolor{blue}{$\diamond$}). Likewise, for the second and third rows in Subdomain 2, we find the approximated value of the solution at (\textcolor{blue}{$\otimes$}) using the same procedure.

\begin{figure}[h]
	\centering
	\includegraphics[width=9cm, height=6cm]{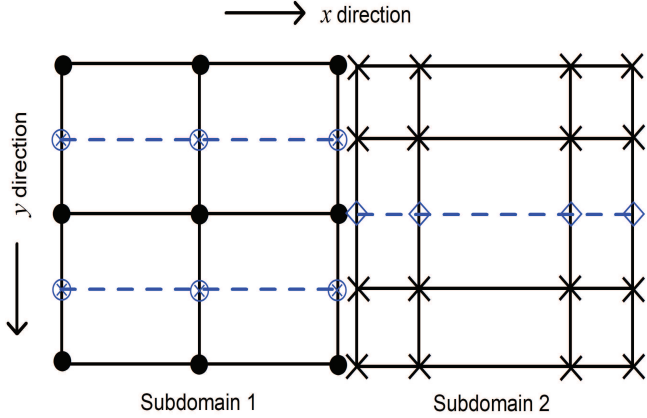}
	\caption{Implementation of interface and boundary conditions. The symbols ($\bullet$) and ($\times$), are the locations of the meshes at each of the two subdomains. The values of the approximated solution are found by interpolatig at points labeled as (\textcolor{blue}{$\diamond$}) and (\textcolor{blue}{$\otimes$}), which are required to apply interface and boundary conditions in the $x$ direction.}
	\label{Fig:twodomainmeshBC}
\end{figure}

\subsubsection{Fitzhugh-Nagumo solutions}
\noindent In this section, we include some simulations to show the evolution of the mesh and the solution. For this, we have considered the domain
$$\Omega=\left\{(x,y)|0\leq x\leq 1000, 0\leq y\leq 1000 \right\}.$$
For the Fitzhugh-Nagumo equations (Eqs. \ref{FHNeq}), the physical parameters are $a=0.3$, $b=0.01$ and
$\epsilon=0.005$. The initial condition is given by
\begin{equation}\label{ic_fhn2d}
 u_0(x,y)=\left\{\begin{array}{cll}\frac{1}{\left(1+e^{4(x-450)}\right)^2}-\frac{1}{\left(1+e^{4(x-445)}\right)^2} & \textrm{for} & y \geq 150 \\
 0 & \textrm{for} & y<350
 \end{array}\right.
\end{equation}
and $v_0(x,y)=0.3$ for $x \leq 445$ and $0$ otherwise.

\noindent Two-dimensional solutions of the FHN type equations can be self sustained propagating waves named spiral waves. Fig. \ref{Fig:FHN_2dspiral} shows a spiral wave solution
for $T=117$ where $N_{ix}=N_{iy}=80$, $N_{ccx}=N_{ccy}=3$ and $N_{cfx}=N_{cfy}=15$. AMCMD1 captures the fast transitions with the finer mesh
and a spiral wave is obtained. Observe that the pulse has an extremely small width compared to the size of the domain. This is not always the case. We
can still have the width of the pulse with a similar size of the domain as in \cite{fe02}. In that situation, the refinement is used to solve the presence of the front of the pulse only.

\begin{figure}[h]
	\centering
	\includegraphics[width=10cm, height=10cm]{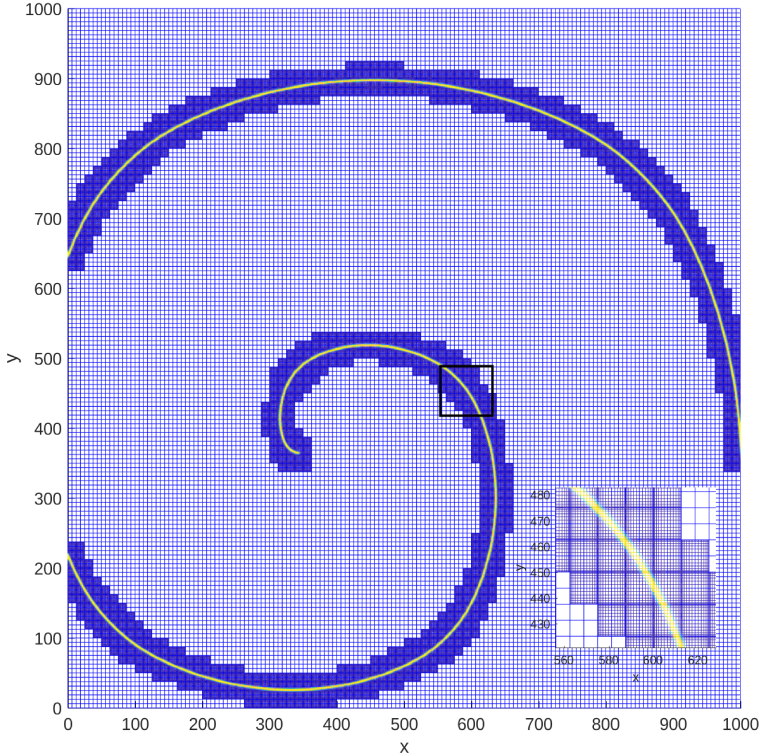}
	\caption{Numerical solution obtained with AMCMD1 for the two- dimensional Fitzhugh-Nagumo equation at $T=117$.}
	\label{Fig:FHN_2dspiral}
\end{figure}

\begin{figure}[h!]
\begin{tabular}{c}
\begin{minipage}{1.5in}
\centering
\includegraphics[width=4.0in,height=1.9in]{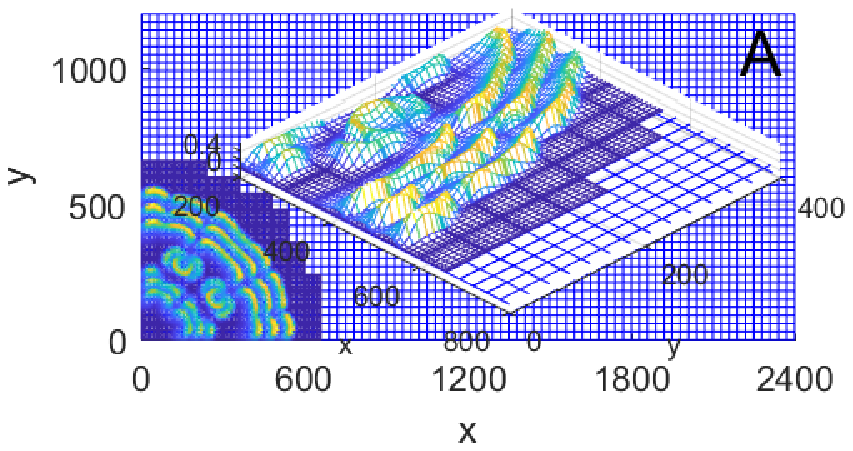}
\end{minipage}
\\
\begin{minipage}{1.5in}
\centering
\includegraphics[width=4.0in,height=1.9in]{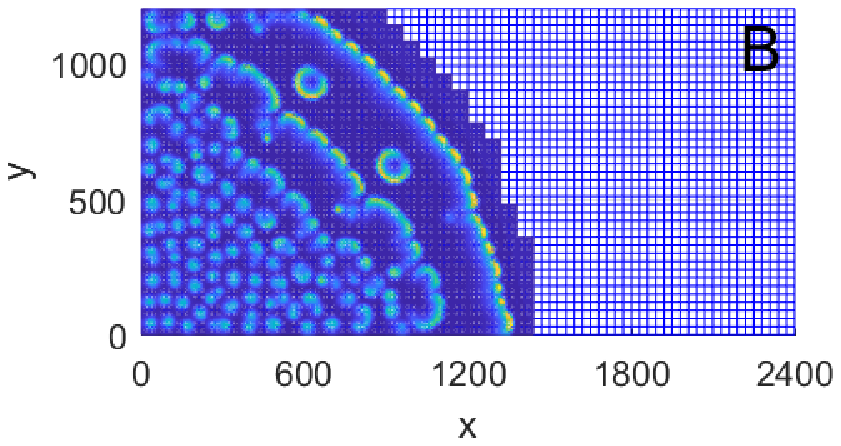}
\end{minipage}
\\
\begin{minipage}{1.5in}
\centering
\includegraphics[width=4.0in,height=1.9in]{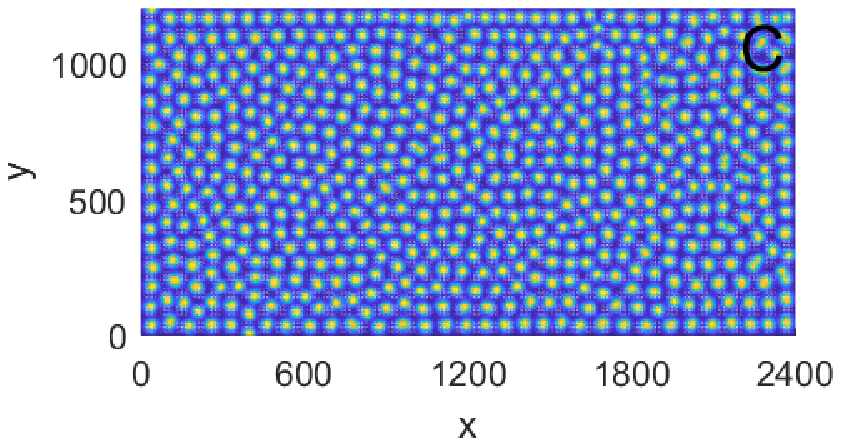}
\end{minipage}
\end{tabular}
	\caption{Numerical solution obtained with AMCMD1 for the two- dimensional Gray-Scott model for integration time A) $T=1600$,
	B) $4400$ and C) $11,200$.}
	\label{Fig:GS_2dpattern}
\end{figure}

\subsubsection{Gray-Scott solutions}

\noindent We finish our work by presenting numerical solutions for the Gray-Scott equations in 2D. In this case,
the main interest is to find time independent patterns. Then, we solve Eq. \ref{GSeq} in 2D ($\Omega=[0,2400]\times [0,1200]$) with parameters $\tau_r=315$, $k_2=1/40$
$b_0=1/15$ and $\delta=7$.
The initial conditions are given by
$$u_0(x)=0.936955$$
and
$$
 v_0(x)=0.1\left(\frac{1}{\left(1+\exp{40(r-150)}\right)^2}-\frac{1}{\left(1+\exp{40(r-100)}\right)^2}\right)+0.014615
$$
with $r=\sqrt{x^2+y^2}$. The obtained patterns are shown in Fig. \ref{Fig:GS_2dpattern} for different $t$ values. In this case, observe that the solution is different from the one dimensional
case, as new pulses are formed. Clearly, by taking $\delta=7$ \cite{pet94} we obtain solitary pulses that bounce at the boundary. However, for the two-dimensional case, we have taken
a circular pulse as an initial condition and therefore propagates from the lower left corner to the rest of the domain. The curvature of the front plays a similar role of increasing the value of $\delta$ for a planar wave and therefore, we obtain the fingering phenomenon for the two-dimensional case \cite{pet94}.

\noindent In Fig. \ref{Fig:GS_2dpattern} we observe that as the pulse propagates from left to right, fast transitions given by the invading pulse and the presence of dots implies that our domain gets completely covered with refined subdomains. This is different to the FHN equations, where the refined domains do not cover the whole domain and just follow the evolution of the spiral wave. The solution in Fig. \ref{Fig:GS_2dpattern} shows the phenomenon of pattern formation due to wave propagation as shown in \cite{Muratov2001}. In this case, the AMCMD method might not be the best choice if we have most of the subdomains in the refined way. Therefore, AMCMD1 might be a good choice when the number of refined subdomains is small compared to the total of subdomains. In the other case, CMD1 may become a better choice.

\section{Discussion and Conclusions}\label{Sec_Discussion}

\noindent In the present work we developed a numerical method that solves the reaction diffusion equations based on Chebyshev polynomials that consider an adaptive mesh. When there
are steep gradients the method automatically refines the grid. The proposed method, AMCMD1, is based on CMD1 in which we follow
a multidomain approach and there is only one common point between subdomains. The method CMD1 presented here for comparisons, has neither been presented elsewhere.

 \noindent AMCMD1 has proven to be much faster than CMD1 in the one dimensional case where the propagation of a single pulse is
studied. AMCMD1 can still be faster than CMD1 if the number of pulses is relatively low. CMD1 may perform better when there are several pulses traveling in the domain, e.g. when there is periodic propagation of pulses such as the ones obtained with the Oregonator \cite{ty94}.

\noindent As discussed in \cite{huang94} our method lies in the category of static adaptive methods, which the authors claim to be inefficient due to the continual readjustment of the mesh. Certainly, AMCMD1 falls into this category; however, as shown in our two dimensional simulations, our method is able to solve equations where multiple pulses
evolve simultaneously, including the birth and death of pulses.

\noindent For the one dimensional pulses obtained with the GS model, we also tested the property of reflection of pulses at the boundary as observed in \cite{pet94}. It is important to
mention that the choice of $\eta$ was crucial to obtain the real physical phenomenon. By taking a large value of $\eta$, the reflection phenomenon of the pulses is not obtained. Therefore, in order to apply the AMCDM1 method it is necessary to adjust some of the parameters before obtaining a physically plausible solution.

\noindent In order to obtain the best mesh for a particular problem, one has to adjust the parameters, $M$, $N_c$, $N_f$, $dt$, $K$ and $\eta$. Once we find a good set of parameters for a domain of
length $L$, it is straightforward to find the parameters for a different lenght $L$.

\noindent In the present work we have presented simulations for different choices of the numerical parameters but we did not focus on the optimal choice of these parameters. It still remains as an open question to find such optimum values.

\noindent This work can be extended by considering at least three levels of refinement of the mesh, instead of the two given in this work with $N_c$ and $N_f$. Clearly from the one dimensional simulations, the vast majority of the domain had absence of a propagating front.
In that case, we can think of splitting the main domain into $N_1$ subdomains and in each of them, have another refinement with refined and coarsed grids. This procedure will reduce even further the computational cost.

\bibliographystyle{siamplain}
\bibliography{biblio_jung_olmos}

\end{document}